\documentclass[11pt]{article}

\usepackage{enumerate}
\usepackage{mathrsfs}
\usepackage{amsmath, color}
\usepackage{latexsym}
\usepackage{amssymb}
\usepackage{amsthm}
\usepackage{amsfonts}
\usepackage{dsfont}
\usepackage{hyperref}
\usepackage{xcolor,cancel}

\newtheorem{dfn}{Definition}[section]
\newtheorem{thm}[dfn]{Theorem}
\newtheorem{rmk}[dfn]{Remark}

\newtheorem{prop}[dfn]{Proposition}
\newtheorem{co}[dfn]{Corollary}
\newtheorem{lem}[dfn]{Lemma}
\theoremstyle{definition}

\numberwithin{equation}{section}

\newcommand{\IE}{{\mathbb{E}}}
\newcommand{\IP}{{\mathbb{P}}}
\newcommand{\IR}{{\mathbb{R}}}
\newcommand{\FF}{{\mathcal{F}}}
\newcommand{\EE}{{\mathcal{E}}}

\newcommand{\IZ}{{\mathbb{Z}}}

\def\wh{\widehat}
\def\wt{\widetilde}
\def\<{\langle}
\def\>{\rangle}

\title{\bf  Discrete Approximation to Brownian Motion with  Darning}
\date{\today}
\author{{\bf Shuwen Lou}}

\begin{document}

\maketitle

\begin{abstract}
Brownian motion with darning (BMD in abbreviation) is introduced and studied in \cite{Chen1} and \cite[Chapter 7]{CF}. Roughly speaking, BMD travels across the ``darning area" at infinite speed, while it behaves like a regular BM outside of this area.  In this paper we show that starting from a single point in its state space, BMD is the weak limit of a family of continuous-time simple random walks on square lattices with diminishing mesh sizes. From any vertex in their state spaces, the approximating random walks jump to its nearest neighbors with equal probability after an exponential holding time. 

\end{abstract}



\section{Introduction}\label{Intro}

Brownian motion with darning  has been introduced and discussed in \cite{Chen1} and \cite[Chapter 7]{CF}. Its definition can be found in, e.g., \cite[Definition 1.1]{Chen1}.    In this paper, let $K\subset \IR^d$ be a compact  connected subset with Lipschitz-continuous boundary. At every $x\in \partial K$, $K$ satisfies the ``cone condition" (see, e.g., \cite[Proposition 1.22]{ChungZhao}),   it is thus  clear that every point on $\partial K$ is regular for $K$ in the sense that $\IP^x[\sigma_K=0]=1$. This allows us to define BMD by  identifying $K$ as a singleton $a^*$ and equipping $E:=(\IR^d\backslash K)\cup \{a^*\}$ with the topology induced from $\IR^d$ (see, e.g., \cite[\S 1.1]{Chen1}). In other words, the distribution of the process on $\IR^d\backslash K$ is the same as regular Brownian motion on $\IR^d$, but the ``darning area" $K$ offers zero resistance to the process. Diffusion processes with darning can be nicely characterized via Dirichlet forms and have been studied with depth in recent literatures, for example, \cite{Chen1, CF, CL}. In particular, we equip $E$ with a measure $m$ which is the same as the Lebesgue measure on $\IR^d\backslash K$, and does not charge $a^*$, Then the BMD on $E$ described above can be characterized by the following Dirichlet form on $L^2(E, m(dx))$:
\begin{equation}\label{def-BMD}
\left\{
\begin{aligned}
\mathcal{D}(\EE)&=   \left\{f: f\in W^{1,2}(\IR^d\backslash K), \, f \text{ is continuous on }E\right\} 
\\
\EE(f,g) &= \frac12 \int_{E}\nabla f(x) \cdot \nabla g(x) m(dx) . 
\end{aligned} 
\right.
\end{equation}

Recently, it was studied in \cite{Lou1, Lou2} how BMD can be approximated by continuous-time random walks on square lattices. However, the discussion in \cite{Lou1, Lou2} was only limited to the toy model of ``Brownian motion with varying dimension" where roughly speaking, the state space is $\IR^2\cup \IR_+$ with a disc  on $\IR^2$ identified with a singleton, i.e., the ``darning point. In this paper, using the method in \cite{Lou2}, we describe how BMD on $\IR^d$ with a darning area $K$ satisfying Lipschitz boundary condition can be approximated by random walks on square lattices. 
The results in this paper provide an intuition for the behavior of BMD upon hitting the ``darning point", and how it is affected by the geometric properties (or intuitively, the ``shape") of the boundary of the darning area.

We now introduce the state spaces of the  approximating random walks.  For every $j\in \mathbb{N}$,  let $K_j:=K\cap 2^{-j}\IZ^d$.  We identify all the  vertices  of $2^{-j}\IZ^d$ that are contained in the compact set $K$ as a singleton $a^*_j$. Let $E^j:=(2^{-j}\IZ^d\cap (\IR^d\backslash K))\cup \{a^*_j\}$.

Recall that in general, a graph $G$  can be written as ``$G=\{G_v, G_e\}$", where $G_v$ is its collection of vertices, and $G_e$ is its connection of edges. Given any two vertices in $a,b\in G$,  if there is unoriented edge with endpoints $a$ and $b$,  we say $a$ and $b$ are adjacent to each other in $G$, written  `` $a\leftrightarrow b$ in $G$".   One can  always assume   that given two vertices $a, b$ on a graph, there is at most  one such unoriented edge connecting these two points (otherwise edges with same endpoints can be removed and replaced with one single edge). This unoriented edge is denoted by $e_{ab}$ or $e_{ba}$ ($e_{ab}$ and $e_{ba}$ are viewed as the same elelment in $G_e$).  In this paper,  for notational convenience, we denote by $\mathcal{G}_j:=\{2^{-j}\IZ^d, \mathcal{V}_j\}$, where $\mathcal{V}_j$ is the collection of the edges of $2^{-j}\IZ^d$.

Next we introduce the graph structure on $E^j$. Denote by $G^j=\{G^j_v, G^j_e\}$ the  graph where $G^j_v=E^j$ is the collection of vertices and $G^j_e$ is the collection of unoriented edges over $E^j$  defined as follows:
\begin{align}
G^j_e:=&\{e_{xy}:\, \exists \,x,y\in 2^{-j}\IZ^d\cap (\IR^d\backslash K), |\,x-y|=2^{-j},\, e_{xy}\in\mathcal{V}_j,\, e_{xy}\cap K=\emptyset\} \nonumber
\\
\cup &\{e_{xa^*_j}:x\in  2^{-j}\IZ^d\cap (\IR^d\backslash K),\,\exists\text{  at least one edge  }e_{xy}\in \mathcal{V}_j \text{ such that }e_{xy}\cap K\neq \emptyset\}. \label{BMD-2}
\end{align}
  Note that $G^j=\{G^j_v, G^j_e\}$ is a connected graph. We emphasize that  given any $x\in G^j_v$, $x\neq a^*_j$, there is at most one unoriented edge in $G^j_e$ connecting $x$ and $a^*_j$.   Denote by $v_j(x)=\#\{e_{xy}\in G^j_e\}$, i.e., the number of vertices in $G^j_v$ adjacent to $x$.

In order to give definition to the approximating random walks for BMD, for every $j\ge 1$, we equip $E^j$  with the measure:
\begin{equation}\label{measure-E^j}
m_j(x):=\frac{2^{-jd}}{2d}\cdot v_j(x), \quad x\in E^j.
\end{equation}
Consider the following  Dirichlet form on $L^2(E^j, m_j)$:
\begin{align}\label{DF-RWVD-form}
\left\{
\begin{aligned}
&\mathcal{D}(\EE^{j})=L^2(E^j, m_j)
\\
&\EE^{j}(f, f)= \frac{2^{-(d-2)j}}{4d}\sum_{\substack{e^o_{xy}:\; e_{xy}\in G^j_e }} \left(f(x)-f(y)\right)^2,
\end{aligned}
\right.
\end{align}
where $e^o_{xy}$ denotes an {\it oriented edge} from  $x $  to  $y$. In other words, given any pair of  adjacent  vertices $x, y \in G^j_v$,  the edge with endpoints $x$ and $y$ is represented twice in the sum: $e^o_{xy}$ and $e^o_{yx}$.  One can verify that $(\EE^j, \mathcal{D}(\EE^{j}))$  is a regular symmetric Dirichlet form  on $L^2(E^j, m_j)$, therefore   there is  symmetric strong Markov process associated with it.  We denote this process   by $X^j$. In \S\ref{S2:Prelim}, we show that each $X^j$ is a continuous-time random walk whose tragectories of $X^j$  stay at each vertex of $E^j$ for an exponentially distributed holding time with parameter $2^{-2j}$ before jumping to one of its neighbors with equal probability.  The main result of this paper is  that the distributions of $\{X^j,\;j\ge 1\}$  converge weakly to the BMD characterized by \eqref{def-BMD}.

The rest of this paper is organized as follows. In \S\ref{S2:Prelim}, we first describe the behavior of $X^j$ by showing their roadmaps. Then we give a brief review on isoperimetric inequalities for weighted graphs, especially the isoperimetric inequalities for $\IZ^d$ equipped with equal weights. Using these results, in Proposition \ref{P:isoperimetric} we prove an isoperimetric inequality for $X^j$.  With the isoperimetric inequality obtained in \S\ref{S2:Prelim},   in \S\ref{S3:Nash} we derive a Nash-type inequality for the family of random walks $\{X^j,\;j\ge 1\}  $, from which we establish heat kernel upper bounds, first on-diagonal then off-diagonal, for the entire family of $\{X^j,\;j\ge 1\}$. In \S\ref{S4:Tightness}, we use the well-known criterion of tightness presented in \cite[Chapter VI, Proposition 3.21]{EK} to prove the tightness of $\{X^j,\;j\ge 1\}$.  The tightness criterion is verified in  Propositions \ref{P:4.6}-\ref{P:4.7}, which are proved using the heat kernel upper bounds obtained in \S\ref{S3:Nash}. Finally, the main result of convergence is given by Theorem \ref{T:main-result}.

In this paper we  follow the convention that in the statements of the theorems or propositions, the capital letters  $C_1, C_2, \cdots$ or $N_1, N_2, \cdots$ denote positive constants or positive integers, whereas in their proofs, the lower letters $c, c_1, \cdots$ or $n_1, n_2, \cdots$  denote positive constants or positive integers whose exact value is unimportant and may change   from line to line.

\section{Preliminaries}\label{S2:Prelim}

\subsection{Roadmap of the approximating random walks}

Suppose $\mathsf{E}$ is a locally compact separable metric space and $\{\mathsf{Q}(x, dy)\}$ is a probability kernel on $(\mathsf{E}, \mathcal{B}(\mathsf{E}))$ with $\mathsf{Q}(x, \{x\})=0$ for every $x\in \mathsf{E}$. Given a constant $\lambda >0$, we can construct a pure jump Markov process $\mathsf{X}$ as follows: Starting from $x_0\in \mathsf{E}$, $\mathsf{X}$ remains at $x_0$ for an exponentially distributed holding time $T_1$ with parameter $\lambda(x_0)$ (i.e., $\IE[T_1]=1/\lambda(x_0)$), then it jumps to some $x_1\in \mathsf{E}$ according to distribution $\mathsf{Q}(x_0, dy)$; it remains at $x_1$ for another exponentially distributed holding time $T_2$  also with  parameter $\lambda(x_1)$ before jumping to $x_2$ according to distribution $\mathsf{Q}(x_1, dy)$.  $T_2$ is independent of $T_1$. $\mathsf{X}$ then continues. The probability kernel $\mathsf{Q}(x, dy)$ is called the {\it roadmap}, i.e., the one-step distribution of $\mathsf{X}$, and the $\lambda(x)$ is its {\it speed function}. If there is a $\sigma$-finite measure $\mathsf{m}_0$ on $\mathsf{E}$ with supp$[\mathsf{m}_0]=\mathsf{E}$ such that
\begin{equation}\label{symmetrizing-meas}
\mathsf{Q}(x, dy)\mathsf{m}_0(dx)=\mathsf{Q}(y, dx)\mathsf{m}_0(dy),
\end{equation}
$\mathsf{m}_0$ is called a {\it symmetrizing measure} of the roadmap $\mathsf{Q}$. 
The following theorem is a restatement of  \cite[Theorem 2.2.2]{CF}.
\begin{thm}[\cite{CF}]\label{DF-pure-jump}
Given a speed function $\lambda >0$. Suppose \eqref{symmetrizing-meas} holds, then the reversible pure jump process $\mathsf{X}$ described above can be characterized by the following Dirichlet form $(\mathfrak{E}, \mathfrak{F})$ on $L^2(\mathsf{E}, \mathsf{m}(x))$ where the underlying reference measure is $\mathsf{m}(dx)=\lambda(x)^{-1}\mathsf{m}_0(dx)$ and
\begin{equation}\label{DF-EN}
\left\{
    \begin{aligned}
        &\mathfrak{F}=  L^2(\mathsf{E},\; \mathsf{m}(x)), \\
        &\mathfrak{E}(f,g) = \frac{1}{2} \int_{\mathsf{E}\times \mathsf{E}} (f(x)-f(y))(g(x)-g(y))\mathsf{Q}(x, dy)\mathsf{m}_0(dx).
    \end{aligned}
\right.
\end{equation}
\end{thm}

With the theorem above, we present the following proposition which states that at every vertex of $E^j$, $X^j$ holds for an exponential amount of time with mean $2^{-2j}$  before jumping to each of its nearest neighbors with equal probability. 

\begin{prop}\label{P:roadmap}
For every $j\ge 1$, $X^j$ has constant speed function $\lambda_j=2^{2j}$ and a roadmap
\begin{equation*}
Q_j(x, dy)= \sum_{\substack{z \in E^j\\ z\leftrightarrow x \text{ in }G^j}} q_j(x, z)\delta_{\{z\}}(dy),
\end{equation*}
where $q_j(x, y)=v_j(x)^{-1}$, for all $x, y\in E^j$.
\end{prop}
\begin{proof}
Define a measure $m^0_j(x):=\lambda_j m_j(x)=2^{-(d-2)j}v_j(x)/(2d)$. The conclusion follows immediately from \eqref{DF-RWVD-form} and  Theorem \ref{DF-EN}.
\end{proof}

\subsection{Isoperimetric inequalities for weighted graphs}

In this section we summarize some results about isoperimetric inequalitie for weighted  graphs in  \cite{MB}. In general, let $\Gamma$ be a locally finite connected graph, and let the collection of vertices of $\Gamma$ be denoted by $\mathbb{V}$. If two vertices $x,y\in \mathbb{V}$ are adjacent to each other, then the the unoriented edge connecting $x$ and $y$ is assigned a unique weight $\mu_{xy}>0$.   Set $\mu_{xy}=0$  if $x$ and $y$ are not adjacent in  $\Gamma$.   Denote by $\mu:=\{\mu_{xy}: x,y \text{ connected in }\Gamma\}$ the assignment of the weights on all the unoriented edges. $(\Gamma, \mu)$ is called a locally finite connected {\it weighted graph}. A weighted  graph  $(\Gamma, \mu)$ can be equipped with the  following  intrinsic   measure $\nu$ on $\mathbb{V}$:
\begin{equation}\label{meas-weighted-graph}
\nu (x):= \sum_{y\in \mathbb{V}: y\leftrightarrow x \text{ in }\Gamma} \mu_{xy}, \quad x\in \mathbb{V}.
\end{equation}
 Given two sets of vertices $A, B$ in $\mathbb{V}$, we define
\begin{equation}\label{def-mu-E}
\mu_\Gamma (A, B):=\sum_{x\in A}\sum_{y\in B}\mu_{xy}.
\end{equation}
The following definition of isoperimetric inequality is taken from \cite[Definition 3.1]{MB}.
\begin{dfn}
For $\alpha\in [1, \infty)$, we say that a weighted graph $(\Gamma, \mu)$ satisfies $\alpha$-isoperimetric inequality ($I_\alpha$) if there exists $C_0>0$ such that
\begin{equation*}
\frac{\mu_\Gamma(A, \mathbb{V}\backslash A)}{\nu(A)^{1-1/\alpha}} \ge C_0, \quad \text{for every finite non-empty }A\subset \mathbb{V}.
\end{equation*}
\end{dfn}

The following proposition  follows from the  combination of \cite[Theorem 3.7, Lemma 3.9, Theorem 3.14]{MB} and the proofs therein. It gives the relationship between Nash-type inequalities and isoperimetric inequalities for weighted graphs.

\begin{prop}[\cite{MB}]\label{MB-iso-implies-nash}
Let $(\Gamma, \mu)$ be a locally finite connected weighted graph satisfying $\alpha$-isoperimetric inequality  with constant $C_0$. Let $\nu$ be the measure defined in \eqref{meas-weighted-graph}. Then $(\Gamma, \mu)$ satisfies the following Nash-type inequality: 
\begin{equation*}
\frac{1}{2}\sum_{x\in \mathbb{V}}\sum_{y\in \mathbb{V}, y\leftrightarrow x} \left(f(x)-f(y)\right)^2\mu_{xy} \ge 4^{-(2+\alpha/2)}C_0^2 \|f\|_{L^2(\nu)}^{2+4/\alpha}\|f\|_{L^1(\nu)}^{-4/\alpha}, \quad f\in L^1(\nu)\cap L^2(\nu)
\end{equation*}
\end{prop}

The next proposition follows immediately from \cite[Theorem 3.26]{MB}. As a notation in \cite{MB}, given a weighted graph $(\Gamma, \mu)$ with collection of vertices $\mathbb{V}$. 
 We denote the  counting measure times $2^{-jd}$ on  $2^{-j}\IZ^d$ by  $\mu_j$, which can be viewed as the measure ``$\nu$" in \eqref{meas-weighted-graph} corresponding to weighted  $2^{-j}\IZ^d$ with all edges weighing $2^{-jd}/2d$.
\begin{prop}[\cite{MB}]\label{iso-Z^2}
For $j\in \mathbb{N}$, let all edges of $2^{-j}\IZ^d$ be assigned with a weight  $2^{-jd}/2d$.    There exists a constant $C_1>0$ independent of $j$ such that  for any finite subset $A$ of $2^{-j}\IZ^d$, 
\begin{equation}\label{eq-iso-Z^d}
\mu_{2^{-j}\IZ^d}(A, 2^{-j}\IZ^d\backslash A)\ge C_1\cdot 2^{-j}  \mu_j (A)^{(d-1)/d}.
\end{equation}
\end{prop}

Before establishing  the isoperimetric inequality for $X_j$, we need the following proposition which will be used throughout this article.

\begin{prop}\label{P1}
There exist  $C_2>0$ and $N_0\in \mathbb{N}$ only depending on  the darning region $K$  such that for all $j\ge N_0$,
\begin{equation}
v_j(a^*_j)\le C_2\cdot 2^{j(d-1)}.
\end{equation}
\end{prop}
\begin{proof}
In the following we denote the $d$- and $(d-1)$-dimensional Lebesgue measures by $m^{(d)}$ and $m^{(d-1)}$, respectively. For any two distinct $x, y\in E^j$  both adjacent to $a^*_j$, the two Euclidean balls $B_{|\cdot|}(x, \frac{2^{-j}}{2})$ and $B_{|\cdot|}(y, \frac{2^{-j}}{2})$ are disjoint. Also, for any $x\leftrightarrow  a^*_j$,  the Euclidean ball $B_{|\cdot|}(x, \frac{2^{-j}}{2})$ must be contained in the set $\{x\in \IR^d:  d_{|\cdot|}(x, \partial K)\le 2\cdot 2^{-j}\}$. Therefore, 
\begin{equation}\label{28}
v_j(a^*_j)\cdot m^{(d)}\left(B_{|\cdot|}\left(x, \frac{2^{-j}}{2}\right)\right)\le m^{(d)}\left(\left\{x\in \IR^d:d_{|\cdot|}(x, \partial K)\le 2\cdot 2^{-j}   \right\}\right).
\end{equation}
Since $K$ has Lipschitz-continuous boundary in $\IR^d$, $\partial K$ is $(d-1)$-dimensional in the sense of both topological and  Minkowski box dimension. This means that there exists $j_0\in \mathbb{N}$ and a constant $c>0$ such that for all $j\ge j_0$, $\partial K$ can be covered by $c\cdot 2^{j(d-1)}$ many boxes with side lenght $2^{-j}$. This further  implies that   set $\{x\in \IR^d:d_{|\cdot|}(x, \partial K)\le 2\cdot 2^{-j} \}$ can be covered by boxes with the same centers but side length $16\cdot 2^{-j}$, which further implies that for some $c>0$,
\begin{equation*}
 m^{(d)}\left(\left\{x\in \IR^d:d_{|\cdot|}(x, \partial K)\le 2\cdot 2^{-j}   \right\}\right) \le c\cdot 2^{j(d-1)}\cdot 2^{-jd} =c\cdot 2^{-j}, \quad j\ge j_0.
\end{equation*}
The conclusion thus follows on account of \eqref{28} and the fact that $m^{(d)}(B(x, 2^{-j}/2))=s(d-1)\cdot 2^{-(j+1)d}$ for all $x\in \IR^d$, where $s(d-1)>0$ is a constant equal to the $(d-1)$-dimensional surface measure. 
\end{proof}

Recall the graph structure on $E^j$ defined in  \eqref{BMD-2}.  In the next proposition we establish an isoperimetric inequality for $X^j$ on the weighted graph $E^j$, where all the edges in $G^j_e$ are equipped with an equal weight of $2^{jd}/(2d)$.

\begin{prop}\label{P:isoperimetric}
For every $j\in \mathbb{N}$,   let all edges of $E^j$   be equipped with an equal weight  $2^{-j}/(2d)$, which is consistent with the definition of $m_j$ in the sense that
\begin{equation*}
m_j(x)=\frac{2^{-jd}}{2d}\cdot \#\left\{e_{xy} \in G^j\right\}.
\end{equation*}
 For the $N_0$ specified in Proposition \ref{P1}, there exist  an  integer $N_1\ge N_0$  and  a constant $C_2>0$ independent of $j$    such that for all $j\ge N_1$,
\begin{equation}\label{iso}
\mu_{E^j}(A, E^j\backslash A)  \ge 2^{-j}C_2m_j(A)^{(d-1)/d}, \quad \text{for any finite set }A\subset E^j.
\end{equation}
\end{prop}

\begin{proof}
Let $A$ be any finite subset of $E^j$. We denote by $K^j:=2^{-jd}\IZ^d\cap K$. Recall that  we denote by $\mathcal{G}_j:=\{2^{-j}\IZ^d, \mathcal{V}_j\}$, where $\mathcal{V}_j$ is the collection of the edges of $2^{-j}\IZ^d$. Also recall that we use ``$e_{xy}$" to denote an edge connecting $x$ and $y$, including these two endpoints.  In the following we establish \eqref{iso} by dividing our discussion into two cases depending on whether $a^*_j$ is in $A$ or not.  
\\
{\it Case (i).} $a^*_j\notin A$. Thus $a^*_j\in E^j\backslash A$ and $A\subset 2^{-j}\IZ^d $.   In view of the definition of the graph structure $G^j$ in \eqref{BMD-2},
\begin{eqnarray}
 \mu_{E^j}(A, \,E^j\backslash A) \nonumber
&=  & \frac{2^{-jd}}{2d} \sum_{x\in A}\#\left\{y\in E^j\backslash A: y\leftrightarrow x \text{ in }G^j \right \}  \nonumber
\\
&= & \frac{2^{}-jd}{2d} \sum_{x\in A } \#\left\{ y\in \left(E^j\backslash \left(A\cup \{a^*_j\} \right)\right): y\leftrightarrow x  \right\} + \frac{2^{-jd}}{2d} \#\left\{ x\in A: x\leftrightarrow a^*_j  \right\} \nonumber
\\
& = & \frac{2^{-jd}}{2d} \sum_{x\in A } \#\left\{ y\in (2^{-j}\IZ^d\backslash A): y\leftrightarrow x\text{ in }2^{-j}\IZ^d,\,  e_{xy}\cap K=\emptyset \right\} \nonumber
\\
& +& \frac{2^{-jd}}{2d}\#\left\{x\in A: \exists \; e_{xy}\in \mathcal{V}_j \text{ such that }e_{xy}\cap K\neq \emptyset \right\}\nonumber
\\
&\ge & \frac{2^{-jd}}{2d} \sum_{x\in A } \#\left\{ y\in (2^{-j}\IZ^d\backslash A):\; \exists \; e_{xy}\in \mathcal{V}_j \text{ such that }   e_{xy}\cap K=\emptyset \right\} \nonumber
\\
&+& \frac{2^{-jd}}{2d}\cdot \frac{1}{2d} \sum_{x\in A} \#\left\{y\in 2^{-j}\IZ^d:\; \exists \; e_{xy}\in \mathcal{V}_j \text{ such that } e_{xy}\cap K \neq \emptyset\right\} \nonumber
\\
&\ge & \frac{2^{-jd}}{4d^2} \sum_{x\in A} \# \left\{ y\in  2^{-j}\IZ^d\backslash A:\; y\leftrightarrow x \text{  in }2^{-j}\IZ^d    \right\}, \nonumber
\end{eqnarray}
where the first inequality above  is due to the fact that for every $x\in A$ such that there is at least one edge in $\mathcal{V}_j$ with an endpoint $x$ intersecting $K$, there are at most $2d$ many such egdes with different other endpoint $y$. Now in view of \eqref{eq-iso-Z^d} and the definition of $\mu_j$ and $\mu_{2^{-j}\IZ^d}$   earlier in this section, we have
\begin{eqnarray}
&& \mu_{E^j}(A, \,E^j\backslash A) \nonumber
\\
&\ge & \frac{2^{-jd}}{4d^2} \sum_{x\in A} \# \left\{ y\in  2^{-j}\IZ^d\backslash A:\; y\leftrightarrow x \text{  in }2^{-j}\IZ^d    \right\} \nonumber 
\\
&\ge & \frac{1}{2d} \cdot   \mu_{2^{-j}\IZ^d } (A, 2^{-j}\IZ^d) 
\stackrel{\eqref{eq-iso-Z^d}}{\ge} \frac{C_1}{2d} \cdot  2^{-j} \mu_j(A)^{(d-1)/d} \label{2.9}
=\frac{C_1}{2d}\cdot  2^{-j} m_j(A)^{(d-1)/d},
\end{eqnarray} 
which the $C_1$ above is the same as in \eqref{eq-iso-Z^d}.
This establishes the desired inequality for the current case. Before  continuing to the other case, we note that by the definition of Lipschitz-continuity,  $\mu_j(K_j)$ is bounded from below by a positive constant for sufficiently large $j$. Thus there exists an integer $j_1\ge N_0$ and a constant $c_1>0$ such that 
\begin{equation}\label{32}
\mu_j(K_j)\ge c_1,\quad \text{for all }j\ge j_1. 
\end{equation}
Now in view of Proposition \ref{P1}, there exists an integer $j_2 \ge N_0$ such that 
\begin{equation}\label{33}
\mu_j(K_j)\ge c_1\ge  \frac{2^{-jd}}{2d}\cdot C_2\cdot 2^{j(d-1)}\ge m_j(a^*_j), \quad \text{for all }j\ge j_2. 
\end{equation}
{\it Case (ii).} $a^*_j\in A$.  In this case $E^j\backslash A = 2^{-j}\IZ^d\backslash A$. Recall that we let $K_j=K\cap 2^{-j}\IZ^d$. Thus for all $j\ge j_2$ given in \eqref{33},
\begin{eqnarray}
&& \mu_{E^j}(A, \,E^j\backslash A) \nonumber
\\
&=  & \frac{2^{-jd}}{2d} \sum_{x\in A}\#\left\{y\in E^j\backslash A: y\leftrightarrow x \text{ in }G^j \right \}  \nonumber
\\
&=& \frac{2^{-jd}}{2d} \sum_{\substack{x\in A\\ x\neq a^*_j}}\#\left\{y\in E^j\backslash A: y\leftrightarrow x \text{ in }G^j \right \} +\frac{2^{-jd}}{2d}\# \left\{y\in E^j\backslash A: y\leftrightarrow a^*_j   \right\} \nonumber
\\
&\ge & \frac{2^{-jd}}{2d} \sum_{\substack{x\in A\\ x\neq a^*_j}}\#\left\{y\in (2^{-j}\IZ^d\backslash A), \; \exists\; e_{xy}\in \mathcal{V}_j:\; e_{xy}\cap K=\emptyset   \right \}  \nonumber
\\
&+& \frac{2^{-jd}}{4d^2} \sum_{x\in 2^{-j}\IZ^d} \#\left\{ y\in (2^{-j}\IZ^d\backslash A), \;\exists \;  e_{xy}\in \mathcal{V}_j:\; e_{xy}\cap K\neq\emptyset     \right\} \nonumber
\\
&\ge & \frac{2^{-jd}}{2d} \sum_{x\in A\backslash \{a^*_j\}}\#\left\{y\in (2^{-j}\IZ^d\backslash A), \; \exists\; e_{xy}\in \mathcal{V}_j:\; e_{xy}\cap K=\emptyset   \right \}  \nonumber
\\
&+&  \frac{2^{-jd}}{4d^2} \sum_{x\in A\backslash \{a^*_j\}}\#\left\{y\in (2^{-j}\IZ^d\backslash A), \; \exists\; e_{xy}\in \mathcal{V}_j:\; e_{xy}\cap K\neq\emptyset   \right \}  \nonumber
\\
&+& \frac{2^{-jd}}{4d^2} \sum_{x\in K_j}\#\left\{ y\in (2^{-j}\IZ^d\backslash A)\; \exists \;e_{xy}\in \mathcal{V}_j  \right\} \nonumber
\\
&\ge & \frac{2^{-jd}}{4d^2} \sum_{x\in (A\backslash \{a^*_j\})\cup K_j} \# \left\{ y\in  2^{-j}\IZ^d\backslash A:\; y\leftrightarrow x \text{  in }2^{-j}\IZ^d    \right\} \nonumber
\\
& = & \frac{1}{2d} \cdot \mu_{2^{-j}\IZ^d } \left((A\backslash \{a^*_j\})\cup K_j,\, 2^{-j}\IZ^d\backslash A   \right) \nonumber
\\
\eqref{eq-iso-Z^d}  & \ge &  \frac{C_1}{2d} \cdot  \mu_j \left((A\backslash \{a^*_j\})\cup K_j\right)^{(d-1)/d}  \stackrel{\eqref{33}}{\ge} \frac{C_1}{2d} \cdot  m_j \left((A\backslash \{a^*_j\})\cup K_j\right)^{(d-1)/d}, \label{2.12}
\end{eqnarray}
where the first inequality is due to the fact that for every $y\in 2^{-j}\IZ^d\backslash A$ such that there is at least one edge  in $\mathcal{V}_j$ with an endpoint $y$ intersecting $K$, there are at most $2d$ many such edges in $\mathcal{V}_j$  with differenet other endpoint $x$. The proof is complete in view of \eqref{2.9} and \eqref{2.12}.
\end{proof}

\section{Nash-type inequality and heat kernel upper bound for random walks on lattices with darning}\label{S3:Nash}

In this section, using the isoperimetric inequality obtained in Proposition \ref{P:isoperimetric}, we establish first a Nash-type inequality and then heat kernel upper bound for $X_j$.
\begin{prop}\label{nash-1}
For every $j\in \mathbb{N}$, let $(P_t^j)_{t\ge 0}$  be  the transition semigroup of $X^j$ with respect to $m_j$.  There exists a constant $C>0$ independent of $j$   such that for all $j\ge N_1$  specified in Proposition \ref{P:isoperimetric}, 
\begin{equation}
\|P^j_t\|_{1\rightarrow \infty}\le \frac{C}{t^{d/2}}, \quad \forall t\in (0, +\infty].
\end{equation}
\end{prop}

\begin{proof}
It follows from \eqref{iso} and Proposition \ref{MB-iso-implies-nash} that
\begin{equation*}
\frac{1}{2}\sum_{x\in E^j}\sum_{y\in E^j, y\leftrightarrow x} \left(f(x)-f(y)\right)^2 \frac{2^{-jd}}{2d} \ge C_2^2\cdot 2^{-2j} \cdot  4^{-2-\frac{2}{d}} \|f\|_{L^2(m_j)}^{2+\frac{4}{d}}\|f\|_{L^1(m_j)}^{-\frac{4}{d}}, \, f\in L^1(m_j)\cap L^2(m_j).
\end{equation*}
In view of the definition of $\EE^j$,  this implies that 
\begin{equation}
\EE^j(f, f)\ge C_2^2\cdot 4^{-2-\frac{2}{d}}\cdot \|f\|_{L^2(m_j)}^{2+4/d}\|f\|_{L^1(m_j)}^{-4/d}, \quad f\in L^1(m_j)\cap L^2(m_j).
\end{equation}
The conclusion now follows from \cite[Theorem 2.9]{CKS}. 
\end{proof}

Now for every $j\in \mathbb{N}$, we define a metric $d_j(\cdot, \cdot)$ on $E^j$ as follows: 
\begin{equation}\label{def-dk}
d_j(x,y):=2^{-j}\times \text{smallest  number of edges  between }x \text{ and }y \text{ in }G^j.
\end{equation}
With the above on-diagonal heat kernel estimate, using the standard Davies's method, we next derive an off-diagonal heat kernel upper bound estimate for $X^j$.  Since there is a Nash-type inequality holds for each $X^j$,  the family of transition density function of  $(P_t^j)_{t\ge 0}$ with respect to $m_j$ exists for every $j\in \mathbb{N}$. We denote this  by $\{p_j(t, x,y),\;  t>0,\; x,y\in E^j \}$.

\begin{prop}\label{general-davies}
For every $j\ge 1$, fix a sequence of $\{\alpha_j\}_{j\ge 1}$ satisfying $\alpha_j\le 2^{j-1}$.     There exists $C_4>0$ independent of $j$, such that for all $j\ge N_1$  specified in Proposition \ref{P:isoperimetric}, 
\begin{equation}
p_j(t,x,y)\le 
\frac{C_4}{t^{d/2}}\exp\left(-\alpha_j d_j(x,y)+4t\alpha_j^2\right),\quad 0<t<\infty, \, x,y\in E^j.
\end{equation}
\end{prop}

\begin{proof}
We prove this result using \cite[Corollary 3.28]{CKS}. For each $j$, we set 
\begin{equation*}
\wh{\FF}^j:=\{h+c:\, h \in \mathcal{D}(\EE^j), \, h\text{ bounded, and }c\in \IR\}.
\end{equation*} 
It is known that the regular symmetric Dirichlet form $(\EE^j, \mathcal{D}(\EE^j))$ is associated with the following energy measure $\Gamma^j$:
$$
\EE^j(u, u)=\int_{E^j}\Gamma^j(u, u), \quad u\in \wh{\FF} ^j.
$$
Now we define $\wh{\FF}^j_\infty$ as a subset of $\psi \in \wh{\FF}^j$ satisfying the following conditions:
\begin{description}
\item{(i)} Both $e^{-2\psi}\Gamma^j(e^{\psi}, e^{\psi})$ and $e^{2\psi}\Gamma^j(e^{-\psi}, e^{-\psi})$  as measures are absolutely continuous with respect to $m_j$ on $E^j$. 
\item{(ii)} Furthermore,
\begin{equation}\label{35}
\Gamma^j(\psi):=\left( \left\|\frac{de^{-2\psi }\;\Gamma^j(e^{\psi}, e^{\psi})}{dm_j} \right\|_\infty \vee  \left\|\frac{de^{2\psi }\;\Gamma^k(e^{-\psi}, e^{-\psi})}{dm_j} \right\|_\infty \right)^{1/2}<\infty.
\end{equation}
\end{description}
For a fixed  constant $\alpha_j \le2^{j-1}$, we denote by
\begin{equation}\label{def-psi-kn}
\psi_{j, n}(x):=\alpha_j \cdot \left(d_j(x, a^*_j)\wedge n\right).
\end{equation}
In order to apply \cite[Corollary 3.28]{CKS}, we need to verify that $\psi_{j,n}\in \wh{\FF}_\infty^j$ for every $n$. Notice that $\psi_{j,n}$ is a constant outside of a bounded domain of $E^j$, therefore it is in $\wh{\FF}^j$. We first note that 
\begin{equation}\label{1}
\left|1-e^x  \right| \le  |2x|, \quad \text{for } -\frac{1}{2}<x<\frac{1}{2}.
\end{equation}
We now first verify conditions (i) and (ii) above.  Viewing  $e^{-2\psi_{j,n}}\Gamma^j(e^{\psi_{j,n}}, e^{\psi_{j,n}})$ as a measure on $E^j$,   given any subset $A\subset E^j$,   we have
\begin{eqnarray}
&&  e^{-2\psi_{j,n}}\Gamma^j(e^{\psi_{j,n}}, e^{\psi_{j,n}}) (A) \nonumber
\\
&=& \frac{2^{-(d-2)j}}{4d} \sum_{x\in E^j\cap A}  e^{-2\psi_{j,n}(x)} \left[ \sum_{y\leftrightarrow x \text{ in }E^j} \left( e^{\psi_{j,n}(y)}-e^{\psi_{j,n}(x)} \right)^2 \right] \nonumber
\\
&\le & \frac{2^{-(d-2)j}}{4d} \sum_{x\in E^j\cap A}\sum_{y\leftrightarrow x \text{ in }E^j} \left[ \left(1-e^{\alpha_j\left( d_j(y, a^*_j)\wedge n -d_j(x, a^*_j)\wedge n \right)}  \right)^2  \right] \nonumber
\\
&=& \frac{2^{-(d-2)j}}{4d} \sum_{\substack{x\in E^j\cap A\\ x\neq a^*_j }}\; \sum_{y\leftrightarrow x \text{ in }E^j} \left[ \left(1-e^{\alpha_j\left( d_j(y, a^*_j)\wedge n -d_j(x, a^*_j)\wedge n \right)}  \right)^2  \right] \nonumber
\\
&+& \frac{2^{-(d-2)j}}{4d} \cdot \mathbf{1}_{\{a^*_j\in A\}} \cdot \sum_{y\leftrightarrow a^*_j}\left[ \left(1-e^{\alpha_j (d_j(y, a^*_j)\wedge n )}  \right)^2  \right] \nonumber
\\
\eqref{1}&\le & \frac{2^{-(d-2)j}}{4d}\cdot \#\left\{x\in E^j\cap A, x\neq a^*_j  \right\}\cdot 2d \cdot (2\cdot \alpha_j\cdot 2^{-j})^2 \nonumber
\\
&+& \frac{2^{-(d-2)j}}{4d} \cdot \mathbf{1}_{\{a^*_j\in A\}} \cdot v_j(a^*_j)\cdot (2\cdot \alpha_j\cdot 2^{-j})^2 \nonumber
\\
&\le & 2\cdot 2^{-jd}\cdot \#\left\{x\in E^j\cap A, x\neq a^*_j  \right\}\cdot \alpha_j^2+\frac{2^{-jd}}{d}\cdot \mathbf{1}_{\{a^*_j\in A\}} \cdot v_j(a^*_j)\cdot \alpha_j^2. \nonumber
\end{eqnarray}
Recall that $m_j(x)=\frac{2^{-jd}}{2d}\cdot v_j(x)$. We conclude that for some $c>0$ independent of $j$, it holds
\begin{equation}\label{34}
\left\|\frac{de^{-2\psi_{j,n} }\;\Gamma^j(e^{\psi_{j,n}}, e^{\psi_{j,n}})}{dm_j} \right\|_\infty \le \sqrt{2} \alpha_j, \quad \text{for all }j\ge 1. 
\end{equation}
Similarly, it can be computed that 
\begin{eqnarray}
&& e^{2\psi_{j,n}}\Gamma^j(e^{-\psi_{j,n}}, e^{-\psi_{j,n}}) (A) \nonumber
\\
&\le & \frac{2^{-(d-2)j}}{4d} \sum_{\substack{x\in E^j\cap A\\ x\neq a^*_j }}\; \sum_{y\leftrightarrow x \text{ in }E^j} \left[ \left(1-e^{\alpha_j\left( d_j(x, a^*_j)\wedge n -d_j(y, a^*_j)\wedge n \right)}  \right)^2  \right] \nonumber
\\
&+& \frac{2^{-(d-2)j}}{4d} \cdot \mathbf{1}_{\{a^*_j\in A\}} \cdot \sum_{y\leftrightarrow a^*_j}\left[ \left(1-e^{-\alpha_j ( d_j(y, a^*_j)\wedge n )}  \right)^2  \right] \nonumber
\\
\eqref{1}  & \le & \frac{2^{-(d-2)j}}{4d}\cdot \#\left\{x\in E^j\cap A, x\neq a^*_j  \right\}\cdot  (2d) \cdot (2\cdot \alpha_j\cdot 2^{-j})^2 \nonumber
\\
&+& \frac{2^{-(d-2)j}}{4d} \cdot \mathbf{1}_{\{a^*_j\in A\}} \cdot v_j(a^*_j)\cdot (2\cdot \alpha_j\cdot 2^{-j})^2 \nonumber
\\
&\le & 2\cdot 2^{-dj}\cdot \#\left\{x\in E^j\cap A, x\neq a^*_j  \right\}\cdot \alpha_j^2+\frac{2^{-dj}}{d}\cdot \mathbf{1}_{\{a^*_j\in A\}} \cdot v_j(a^*_j)\cdot \alpha_j^2. \nonumber
\end{eqnarray}
Similar to \eqref{34}, this shows that  
\begin{equation*}
\left\|\frac{de^{2\psi_{j,n} }\;\Gamma^j(e^{\psi_{j,n}}, e^{\psi_{j,n}})}{dm_j} \right\|_\infty \le \sqrt{2}\alpha_j, \quad \text{for all }j\ge 1,
\end{equation*}
and thus \eqref{35} is verified. The desired conclusion follows immediately from  \cite[Theorem 3.25, Corollary 3.28]{CKS}.

\end{proof}

\begin{co}\label{HKUB}
There exist $C_5>0$ independent of $j$ such that for all $j\ge N_1$ specified in Proposition \ref{P:isoperimetric}, all $x,y\in E^j$ and all $t\ge 0$, it holds
\begin{equation*}
p_j(t,x,y)\le \left\{
\begin{aligned}
&\frac{C_5}{t^{d/2}}e^{-d_j(x,y)^2/(64t)}, & \text{ when }d_j(x,y)\le 16\cdot 2^jt;
\\
&\frac{C_5}{t^{d/2}} e^{-2^j d_j(x,y)/4},
& \text{ when }d_j(x,y)\ge 16\cdot 2^jt.
\end{aligned}
\right.
\end{equation*}
In particular, given any $T>0$, there exists $C_6>0$ such that
\begin{equation*}
p_j(t, x, y)\le \frac{C_6}{t^{d/2}}\left( e^{-d_j(x,y)^2/(64t)} + e^{-2^j d_j(x,y)/4} \right), \, \text{for all }(t, x, y)\in (0, T]\times E^j\times E^j.
\end{equation*}
\end{co}
\begin{proof}
To prove this, in Proposition \ref{general-davies}, given any $j\ge N_1$, for any fixed $t_0>0$ and any pair of $x_0, y_0\in E^j$, we take 
\begin{equation*}
\alpha_j:=\frac{d_j(x_0,y_0)}{32t_0}\wedge 2^{j-1}.
\end{equation*}
Then Proposition \ref{general-davies} yields that  for all $t>0$ and $x,y\in E^j$, 
\begin{align*}
p_j(t_0,x, y)\le \frac{c}{t_0^{d/2}}\exp\left[-\left(\frac{d_j(x_0,y_0)}{32t_0}\wedge 2^{j-1}\right) d_j(x,y)+4t_0\left( \frac{d_j(x_0,y_0)}{16t_0}\wedge 2^{j-1}   \right)^2\right].
\end{align*}
The desired result follows from first taking $x=x_0$ and $y=y_0$, then  simplying the right hand side above by dividing it into two cases: $d_j(x_0, y_0)\ge 32t_0\cdot 2^{j-1}$ and  $d_j(x_0, y_0)\le 32t_0\cdot 2^{j-1}$.
\end{proof}

\section{Tightness of the approximating random walks}\label{S4:Tightness}

The next proposition taken from \cite{JS} is a well-known  criterion for tightness for c\`{a}dl\`{a}g processes. As a standard notation, given a metric $d(\cdot,\, \cdot)$, we denote by
\begin{equation*}
w_d(x,\, \theta,\, T):=\inf_{\{t_i\}_{1\le i\le n}\in \Pi} \max_{1\le i\le n} \sup_{s, t\in [t_i, t_{i-1}]} d(x(s), x(t)),
\end{equation*}
where $\Pi$ is the collection of all possible partitions of the form $0=t_0<t_1<\cdots <t_{n-1}<T\le t_n$ with $\min_{1\le i\le n} (t_i-t_{i-1})\ge \theta$ and $n\ge 1$.
\begin{prop}[Chapter VI, Theorem 3.21 in \cite{JS}]\label{tightness-criterion}
Let $\{Y_k, \IP^y\}_{k\ge 1}$ be a a sequence of c\`{a}dl\`{a}g processes on state space $E$.  Given $y\in E$, the laws of $\{Y_k, \IP^y\}_{k\ge 1}$ are tight in the Skorokhod space $D([0, T],E, \rho)$  if and only if
\begin{description}
\item{(i).} For any $T>0$, $\delta>0$, there exist $K_1\in \mathbb{N}$ and $M>0$ such that for all $k\ge  K_1$,
\begin{equation}
\IP^y\left[\sup_{t\in [0, T]}\big|Y^k_t\big|_\rho>M\right]<\delta. 
\end{equation} 
\item{(ii).} For any $T>0$, $\delta_1, \delta_2>0$, there exist $\delta_3>0$ and $K_2>0$ such that for all $k\ge K_2$,
\begin{equation}
\IP^y\left[w_\rho\left( Y^k  ,\, \delta_3,\, T\right)>\delta_1\right]<\delta_2. 
\end{equation}
\end{description}
\end{prop}

In this section, we use the heat kernel upper bounds obtained in  Corollary \ref{HKUB} to verify the two conditions in Proposition \ref{tightness-criterion}.  We begin with the following two lemmas that can be proved in the same  manner as those in \cite[Lemma 4.2, Proposition 4.3]{Lou2}. In particular, Lemma \ref{L:4.2} is a standard result due to the strong Markov property of $X^j$, and Lemma \ref{L:4.3} results from the fact that $X^j$ making jumps at a rate of $2^{-2j}$, for each $j\ge 1.$  We skip the proofs of these two lemmas.

\begin{lem}\label{L:4.2}
Given any $T, M>0$, for any sufficiently large $j\in \mathbb{N}$ such that $2^{-j}<T$, it holds for all $x\in E^j$  that
\begin{align}
\IP^x\left[\sup_{t\in [0, T]}|X^j_t|_\rho\ge M\right]&\le \IP^x\left[   \sup_{t\in [0, 8^{-j}]} |X^j_t|_\rho\ge M \right]+\IP^x\left[\left|X_T^j\right|_\rho \ge \frac{M}{2}\right] \nonumber
\\
& +\IP^x\left[ T-8^{-j}\le \tau_M  \le T,\, \left|X^j_T\right|_\rho\le \frac{M}{2}\right]  \nonumber
\\
&+\IP^x\left[8^{-j}\le \tau_M\le T-8^{-j},\, \left|X^j_T\right|_\rho\le  \frac{M}{2} \right],  \nonumber
\end{align}
where $\tau_M:=\inf\{t>0: \;|X^j_t|_\rho\ge M\}$.
\end{lem}

\begin{lem}\label{L:4.3}
For any $\delta>0$, any $T>0$, there exists $M_1>0$ such that for all $j\ge 1$:
\begin{equation*}
\sup_{y\in E^j}\IP^y\left[ \sup_{t\in [0, 8^{-j}]} \rho\left(X_0^j, X_t^j\right)\ge M_1\right]<\delta.
\end{equation*}
\end{lem}

Given any $r>0$, we denote  the boundary of a  ``cube" in $\IR^d$ centered at the origin with side length $2r$ by
\begin{equation}\label{def-Sk}
S_r := \left\{ (x_1, \dots , x_d)\in \IR^d: \max_{1\le i\le d} |x_i|=r\right\}.
\end{equation}
For the remaining of this paper, we fix a starting point $x_0\in \cap_{j\ge 1}E^j$ and a $k_0\in \mathbb{N}$, such that both $K$ and $x_0$ are contained in the set
\begin{equation}\label{def-k0}
\left\{ (x_1, \dots , x_d)\in \IR^d: \max_{1\le i\le d} |x_i|\le k_0\right\}.
\end{equation}
By elementary geometry, it can be told that for all $r\ge 2k_0,\, j\ge 1$,  
\begin{equation}\label{4}
\# \left(S_r\cap E^j\right) \le (2d)\cdot (2r\cdot 2^j)^{d-1} .
\end{equation}
When $k\ge 2k_0$, for all $j\ge 1$, the definition of $d_j$ implies that
\begin{equation}\label{3}
d_j(x_0, y)\ge \rho(x_0, y)\ge \frac{k}{2}, \quad \text{for }y\in S_k\cap E^j.
\end{equation}

\begin{lem}\label{L:4.4}
Fix $x_0\in \bigcap_{j\ge 1}E^j$  and $T>0$.  For any $\delta>0$, there exists $M_2>0$  such that for all $j\ge N_1$ specified in Proposition \ref{P:isoperimetric} such that $e^{-j}<T$: 
\begin{equation*}
\sup_{8^{-j}\le t\le T}\IP^{x_0}\left[d_j(X_t^j, x_0)\ge M_2\right]<\delta.
\end{equation*}
\end{lem}

\begin{proof}
We use Proposition \ref{HKUB} to prove this.  We first note that the sequence of  metrics $\{d_j\}_{j\ge 1}$ is non-increasing in $j$. Therefore, for the $k_0$ specified in \eqref{def-k0}, one can choose $M>2k_0$ sufficiently large such that 
\begin{equation}\label{35}
 d_j(a^*_j, x_0) \le M , \quad \text{for all }j\ge 1.
\end{equation}
For $M$ satisfying \eqref{35}, in view of the definition of $m_j$ and the heat kernel upper bound in Corollary \ref{HKUB}, there exists some $c_1>0$  independent of $j$  such that 
\begin{eqnarray}
&&\IP^{x_0}\left[ d_j(X_t^j, x_0)\ge M\right]  \nonumber
\\
&\le & \sum_{d_j(y, x_0)\ge M} \frac{c_1}{t^{d/2}}\left(e^{-\frac{d_j(x_0, y)^2}{64t}}+ e^{-\frac{2^jd_j(x_0, y)}{4}}\right) m_j(dy). \nonumber
\\
&\le &  \sum_{d_j(y, x_0)\ge M} \frac{c_1}{t^{d/2}}e^{-\frac{d_j(x_0, y)^2}{64t}} \cdot 2^{-jd} +  \sum_{d_j(y, x_0)\ge M}\frac{c_1}{t^{d/2}}e^{-\frac{2^jd_j(x_0 ,y)}{4}} \cdot 2^{-jd}. \label{2}
\end{eqnarray}
To give an upper bound for each of the two terms on the right hand side above, we first compute for  a generic $k\ge 0$,
\begin{eqnarray}
&& \sum_{l=k\cdot 2^j}^\infty \;\sum_{y\in E^j\cap S_{2k_0+l\cdot 2^{-j}}} \frac{1}{t^{d/2}}2^{-jd}\cdot e^{-\frac{d_j(x_0, y)^2}{64t}} \nonumber
\\
\eqref{3} &\le & \sum_{l=k\cdot 2^j}^\infty \;\sum_{y\in E^j\cap S_{2k_0+l\cdot 2^{-j}}}  \frac{1}{t^{d/2}}2^{-jd}\cdot e^{-\frac{(2k_0+l\cdot 2^{-j})^2}{256t}} \nonumber
\\
\eqref{4} &\le & \sum_{l=k\cdot 2^j}^\infty  \frac{1}{t^{d/2}}2^{-jd}\cdot e^{-\frac{(2k_0+l\cdot 2^{-j})^2}{256t}}\cdot (2d)\cdot (4k_0+2l\cdot 2^{-j})^{d-1} \cdot 2^{j(d-1)} \nonumber
\\
&\le & \sum_{l=k\cdot 2^j}^\infty  \frac{2d\cdot 2^{-j}}{t^{d/2}}\cdot (4k_0+2l\cdot2^{-j})^{d-1} \cdot e^{-\frac{(2k_0+l\cdot 2^{-j})^2}{256t}} \nonumber
\\
&\le & \sum_{l=k\cdot 2^{j}}^\infty 2d\cdot 2^{-j}\cdot (4k_0+2l\cdot 2^{-j})^{d-1} \cdot e^{-\frac{(2k_0+l\cdot 2^{-j})^2}{512T}}\cdot \left(\sup_{0<t\le T}\frac{1}{t^{d/2}}e^{-\frac{(2k_0)^2}{512t}}\right) \nonumber
\\
&\le & c_2 \sum_{l=k\cdot 2^j}^\infty  2^{-j}\cdot (4k_0+2l\cdot 2^{-j})^{d-1} \cdot e^{-\frac{(2k_0+l\cdot 2^{-j})^2}{512T}} \nonumber
\\
&\le & c_2 \sum_{u=k}^\infty \sum_{l=2^j\cdot u}^{2^j(u+1)-1} 2^{-j}\cdot (4k_0+2l\cdot 2^{-j})^{d-1} \cdot e^{-\frac{(2k_0+l\cdot 2^{-j})^2}{512T}} \nonumber
\\
&\le & c_2 \sum_{u=k}^\infty (4k_0+2u+2)^{d-1}e^{-\frac{(2k_0+u)^2}{512T}}\nonumber
\\
&\le & c_2 \sum_{u=k}^\infty (4k_0+2u+2)^{d-1}e^{-\frac{(2k_0+u)^2}{512T}}\cdot e^{-\frac{(2k_0+u)^2}{512T}} \nonumber
\\
&\le & c_2 \left(\sup_{x\ge 2k_0}(2x+2)^{d-1}e^{-\frac{x^2}{512T}}\right)\sum_{u=k}^\infty e^{-\frac{(2k_0+u)^2}{512T}}
\le  c_2 \sum_{u=k}^\infty e^{-\frac{(2k_0+u)^2}{512T}}.
\label{5}
\end{eqnarray} 
It now follows that for any $\delta>0$, there exists $k_1\in \mathbb{N}$ such that for all $k\ge k_1$,
\begin{equation}\label{8}
\sum_{l=k\cdot 2^j}^\infty \;\sum_{y\in E^j\cap S_{2k_0+l\cdot 2^{-j}}}  \frac{1}{t^{d/2}}2^{-jd}\cdot e^{-\frac{d_j(x_0, y)^2}{32t}} \le c_1 \sum_{l=k}^\infty e^{-\frac{(2k_0+l)^2}{512T}}<\delta/2. 
\end{equation}
For the second term on the right hand side of \eqref{2}, similarly we have for any generic $k\ge 0$,
\begin{eqnarray}
&&  \sum_{l=k\cdot 2^j}^\infty \;\sum_{y\in E^j\cap S_{2k_0+l\cdot 2^{-j}}} \frac{1}{t^{d/2}}e^{-\frac{2^jd_j(x_0 ,y)}{4}} \cdot 2^{-jd} \nonumber
\\
\eqref{3}, \eqref{4}  &\le & \sum_{l=k\cdot 2^j}^\infty \frac{1}{t^{d/2}} e^{-\frac{2^j (2k_0+l\cdot 2^{-j})}{8}} (2d)\cdot (4k_0+2l\cdot 2^{-j})^{d-1} \cdot 2^{-j}\nonumber
\\
(t\ge 8^{-j})&  \le &   2d\sum_{l=k\cdot 2^j}^\infty 2^{-j}  (4k_0+2l\cdot 2^{-j})^{d-1}e^{-\frac{2^j (2k_0+l\cdot 2^{-j})}{16}} \left( \sup_{j\ge 1} 8^{\frac{jd}{2}}e^{-\frac{2^j}{16}}  \right) \nonumber
\\
&\le & c_3 \sum_{l=k\cdot 2^j}^\infty 2^{-j}  (4k_0+2l\cdot 2^{-j})^{d-1}e^{-\frac{2^j (2k_0+l\cdot 2^{-j})}{16}}\nonumber
\\
&\le & c_3\sum_{u=k}^\infty \sum_{l=2^j\cdot u}^{2^j(u+1)-1} 2^{-j}(4k_0+2l\cdot 2^{-j})^{d-1}e^{-\frac{2^j (2k_0+l\cdot 2^{-j})}{16}}\nonumber
\\
&\le &  c_3\sum_{u=k}^\infty (4k_0+2u+2)^{d-1}e^{-\frac{2^j(2k_0+u)}{16}}\nonumber
\\
&\le &  c_3 \left(\sup_{x\ge 2k_0} (2x+2)^{d-1}e^{-\frac{x}{16}}\right) \sum_{u=k}^\infty e^{-\frac{2^j(2k_0+u)}{16}} \le c_1\sum_{u=k}^\infty e^{-\frac{2k_0+u}{16}}.\label{6}
\end{eqnarray}
Therefore, for any $\delta>0$, there exists $k_2\in \mathbb{N}$ such that for all $k\ge k_2$,
\begin{equation}\label{7}
 \sum_{l=k\cdot 2^j}^\infty \;\sum_{y\in E^j\cap S_{2k_0+l\cdot 2^{-j}}} \frac{1}{t^{d/2}}e^{-\frac{2^jd_j(x_0 ,y)}{2}} \cdot 2^{-jd} \le c_1\sum_{u=k}^\infty e^{-\frac{2k_0+u}{16}}<\delta_2. 
\end{equation}
Finally, we note that for any $j\ge 1$, $E^j\subset \cup_{l=0}^\infty S_{l\cdot 2^{-j}}$. Thus for the $k_1, k_2$ specified in \eqref{8} and \eqref{7}, there exists $M>0$ sufficiently large such that 
\begin{equation}
\left\{y\in E^j: d_j(x_0, y)\ge M\right\} \subset  \bigcup_{l=(k_1\vee k_2)\cdot 2^j}^\infty \left( E^j\cap S_{2k_0+l\cdot 2^{-j}}  \right) \quad \text{for all }j\ge 1. 
\end{equation}
This combining with \eqref{2}, \eqref{8}, and \eqref{7} implies that given $\delta>0$, for such chosen $M>0$, 
\begin{equation}
\IP^{x_0}\left[ d_j(X_t^j, x_0)\ge M\right]<\delta, \quad \text{for all }j\ge 1. 
\end{equation}
\end{proof}

\begin{lem}\label{L:4.5}
For any fixed  $T>0$, $\delta>0$, there exist $M_3>0$ and an integer  $N_2\ge N_1$ satisfying    $ T>2\cdot 8^{-N_2}$, such that for all $j\ge N_2$ and $M>M_3$,  it holds that
\begin{equation*}
\IP^{y}\left[\big|X^j_t\big|_\rho\le \frac{M}{2}\right]<\delta, \quad \text{for all }|y|_\rho>M,\, t\in [8^{-j}, T-8^{-j}].
\end{equation*} 
\end{lem}

\begin{proof}
For a generic $M>0$, given $|y|_\rho>M$ and $t\in  [8^{-j}, T-8^{-j}]$,
\begin{eqnarray}
\IP^{y}\left[\big|X^j_t\big|_\rho\le \frac{M}{2}\right] & =& \sum_{\substack{\rho(x, a^*_j)\le M/2\\ x\neq a^*_j}} p_j(t, y, x)m_j(x)+p_j(t,y, a^*_j)m_j(a^*_j) \label{12}
\end{eqnarray}
In view of Proposition \ref{P1} and Corollary \ref{HKUB}, since $d_j(y, a^*_j)\ge \rho(y, a^*_j)\ge M$ and $t\in[8^{-j}, T-8^{-j}]$, there exists $c_1>0$ such that
\begin{eqnarray}
p_j(t,y, a^*_j)m_j(a^*_j) &\le & \frac{c_12^{-j}}{t^{d/2}}\left( e^{-\frac{M^2}{64t}} + e^{-\frac{2^j M}{4}} \right) \le c_1\left( \frac{1}{t^{d/2}}e^{-\frac{M^2}{64t}}+\frac{2^{-j}}{8^{-jd/2}}e^{-\frac{2^jM}{4}}\right). \label{10}
\end{eqnarray}
Given any $\delta>0$, there exists $c_2>0$ such that when $M\ge c_2$, 
\begin{equation*}
\sup_{0<t\le T} \frac{c_1}{t^{d/2}}e^{-\frac{M^2}{64t}}<\frac{\delta}{4}.
\end{equation*}
For the second term on the right hand side of \eqref{10}, given any $\delta>0$,  there exist  $k_1\in \mathbb{N}$ and $c_3>0$   such that for all $j\ge k_1$ and all $M>c_3$ speicified above, 
\begin{equation*}
\frac{c_12^{-j}}{8^{-jd/2}}e^{-\frac{2^jM}{2}} =c_1 (2^j)^{\frac{3d}{2}-1}e^{-\frac{M\cdot2^j}{2}}<\frac{\delta}{4}. 
\end{equation*}
Thus we have showed that given any $\delta>0$, there exist $c_4>0$ and $k_2\in \mathbb{N}$ such that for all $M>c_4$ and all $j\ge k_1$, 
\begin{equation}\label{11}
p_j(t,y, a^*_j)m_j(a^*_j)<\frac{\delta}{2}.
\end{equation}
For the other term on the right hand side of \eqref{12}, first we note that since  $K\subset \{(x_1, \dots , x_d)\in \IR^d: \max_{1\le i\le d}|x_i|\le k_0 \}$, 
\begin{equation*}
\left\{x\in E^j: \rho(x, a^*_j) \le \frac{M}{2} \right\} \subset \left\{(x_1, \dots , x_d)\in \IR^d:\max_{1\le i\le d} |x_i|\le k_0+2M  \right\}.
\end{equation*}
This implies that there exists some $c_5>0$ such that for $M\ge 2k_0$,
\begin{equation}\label{36}
\#\left\{x\in E^j: \rho(x, a^*_j) \le \frac{M}{2} \right\} \le c_5\cdot M^d\cdot 2^{jd}.
\end{equation}
Again since $d_j(y, x)\ge \rho(y,x)\ge M$ and $t\in[8^{-j}, T-8^{-j}]$, by Corollary \ref{HKUB},
\begin{eqnarray}
\sum_{\substack{\rho(x, a^*_j)\le M/2\\ x\neq a^*_j}} p_j(t, y, x)m_j(x) & \le & \sum_{\substack{\rho(x, a^*_j)\le M/2\\ x\neq a^*_j}}c_6 \left( \frac{1}{t^{d/2}}e^{-\frac{M^2}{256t}}+\frac{1}{8^{-jd/2}}e^{-\frac{2^jM}{8}} \right) 2^{-jd} \nonumber
\\
\eqref{36}   &\le & c_6\cdot M^d \cdot 2^{jd}\left(\frac{1}{t^{d/2}}e^{-\frac{M^2}{256t}}+8^{jd/2}e^{-\frac{2^jM}{8}} \right)\cdot 2^{-jd} \nonumber
\\
(M\ge 2k_0)  &\le &  c_6M^de^{-\frac{M^2}{512T}}\left( \sup_{0<t\le T} e^{-\frac{4k_0^2}{512T}}\right)  + c_6\cdot M^d\cdot  2^{\frac{3jd}{2}}e^{-\frac{2^jM}{8}} \nonumber
\\
&\le &  c_6M^de^{-\frac{M^2}{512T}}  + c_6\cdot M^d\cdot  2^{\frac{3jd}{2}}e^{-\frac{2^jM}{8}}.
\label{38}
\end{eqnarray}
Given any $\delta>0$, for the first term on the right hand side of \eqref{38}, there exists $c_7\ge 2k_0$ such that for all $M\ge c_7$, $c_6\cdot M^d\cdot e^{-\frac{M^2}{512T}} <\delta/4$.   For the second term on the right hand side of \eqref{12}, we first denote by $c_8:= \sup_{M\ge c_7} M^de^{-M/16}$. Thus for all $M\ge c_7$,
\begin{equation}
c_6\cdot M^d 2^{\frac{3jd}{2}}e^{-\frac{2^jM}{8}} \le c_6\cdot M^d 2^{\frac{3jd}{2}}e^{-\frac{2^jM}{16}}\cdot e^{-\frac{M}{16}} \le c_6c_8 \cdot 2^{\frac{3jd}{2}} e^{-\frac{2^jc_7}{16}}.
\end{equation}
From the last display above, one can tell that  there exists $k_3\in \mathbb{N}$ such that for all $j\ge k_3$ and all $M\ge c_1$,
\begin{equation}\label{37}
c_6\cdot M^d\cdot  2^{\frac{3jd}{2}}e^{-\frac{2^jM}{8}} <\frac{\delta}{4}.
\end{equation} 
Combining \eqref{37} with the previous discussion regarding the first term on the right hand side of \eqref{12}, it has been shown that there exist $c_7>0$ and $k_3\in \mathbb{N}$ such that for all $M\ge c_7$ and all $j\ge k_3$
\begin{equation}\label{13}
\sum_{\substack{\rho(x, a^*_j)\le M/2\\ x\neq a^*_j}} p_j(t, y, x)m_j(x) <\frac{\delta}{4}+\frac{\delta}{4}=\frac{\delta}{2}.
\end{equation}
Finally, combining \eqref{12}, \eqref{11}, and  \eqref{13}, it has been shown that for all $M\ge c_4\vee c_7$ and all $j\ge k_2\vee k_3$, 
\begin{equation*}
\IP^{y}\left[\big|X^j_t\big|_\rho\le \frac{M}{2}\right]<\delta, \quad \text{for all }|y|_\rho>M,\, t\in [8^{-j}, T-8^{-j}].
\end{equation*} 
\end{proof}

\begin{lem}\label{L:4.6}
Fix $x_0\in \bigcap_{j\ge 1}E^j$. For every $T>0$ and every $\delta>0$, for the $N_2$   and $M_3$ given in Lemma \ref{L:4.5}, it holds for  all $j\ge N_2$ and all $M>M_3$  that
\begin{equation*}
\IP^{x_0}\left[8^{-j}\le \tau_M\le T-8^{-j},\, \left|X^j_T\right|_\rho\le  \frac{M}{2} \right]<\delta.
\end{equation*}
\end{lem}
\begin{proof}
\begin{eqnarray}
&&\IP^{x_0}\left[8^{-j}\le \tau_M\le T-8^{-j},\, \left|X^j_T\right|_\rho\le  \frac{M}{2} \right]  \nonumber
\\
&=&\int_{8^{-j}}^{T-8^{-j}}\IE^{x_0}\left[ \IP^{X^j_s}\left[\left|X^j_{T-s}\right|\le \frac{M}{2}\right]; \tau_M\in ds\right] \nonumber
\\
&\le &  \int_{8^{-j}}^{T-8^{-j}} \IE^{x_0}\left[ \sup_{\substack{y: |y|_\rho\ge M \\ t\in [8^{-j}, T-8^{-j}}]} \IP^{y}\left[|X^j_t|_\rho\le \frac{M}{2} \right];\tau_M\in ds \right] \nonumber
\\
& \le &  \sup_{\substack{|y|_\rho\ge M\\ t\in [8^{-j}, T-8^{-j}]}}\IP^y\left[|X^j_t|_\rho \le \frac{M}{2}\right].
\end{eqnarray}
The conclusion then follows from Lemma \ref{L:4.5}.
\end{proof}

\begin{prop}\label{P:4.6}
Given $x_0 \in  \cap_{j\ge 1} E^j$,  given any $T>0$, $\delta>0$, there exist $M_4>0$ and an integer $N_3\ge N_1$,  such that for all $j\ge N_3$,
\begin{equation}
\IP^{x_0}\left[\sup_{t\in [0, T]}\big|X^j_t\big|_\rho>M_4\right]<\delta. 
\end{equation} 
\end{prop}

\begin{proof}
On account of Lemma \ref{L:4.2}, it suffices to show that given any $\delta>0$, $T>0$, there exist  $c_1>0$ and  $n_1\in \mathbb{N}$ with $8^{-n_1}<T/2$,  such that for all $M>c_1$ and all $j\ge n_1$, the following hold:
\begin{description}
\item{(i)} $\IP^{x_0}\left[ \sup_{t\in [0, 8^{-j}]} |X^j_t|_\rho\ge M \right]<\delta$;
\item{(ii)} $\IP^{x_0}\left[ \left|X_T^j \right|_\rho \ge \frac{M}{2}  \right]<\delta$.
\item{(iii)} $\IP^{x_0}\left[ T-8^{-k}\le \tau_M  \le T,\, \left|X^j_T\right|_\rho\le \frac{M}{2}\right]<\delta $;
\item{(iv)} $\IP^{x_0}\left[8^{-j}\le \tau_M\le T-8^{-j},\, \left|X^j_T\right|_\rho\le  \frac{M}{2} \right]<\delta$.
\end{description}

We claim that  all (i)-(iv) are satisfied for all $M\ge |x_0|_\rho+2\cdot \max_{1\le i\le 4}M_i$ and $j\ge \max_{1\le i\le 3}N_i$, where the $M_i$'s and the $N_i$'s are given in Lemmas \ref{L:4.3}, \ref{L:4.4} and \ref{L:4.6}. Actually it is  evident that (i) and (ii) hold on account of Lemma \ref{L:4.3} and \ref{L:4.4}, respectively, together with triangle inequality for distances.  (iv) holds in view of Lemma \ref{L:4.6}. To justify (iii), by the same argument as that for \cite[(4.18)]{Lou2} we have
\begin{eqnarray}
&&\IP^{x_0}\left[ T-8^{-j}\le \tau_M  \le T,\, \left|X^j_T\right|_\rho\le \frac{M}{2}\right] \le  \sup_{|y|_\rho\ge M} \IP^y\left[\sup_{t\in [0, 8^{-j}]}   \rho\left(X^j_t, X^j_0\right)\geq\frac{M}{2}\right].
\end{eqnarray}
Thus (iii) follows from Lemma \ref{L:4.5}.
\end{proof}

As a standard notation, given a metric $d(\cdot,\, \cdot)$, we denote by
\begin{equation*}
w_d(x,\, \theta,\, T):=\inf_{\{t_i\}_{1\le i\le n}\in \Pi} \max_{1\le i\le n} \sup_{s, t\in [t_i, t_{i-1}]} d(x(s), x(t)),
\end{equation*}
where $\Pi$ is the collection of all possible partitions of the form $0=t_0<t_1<\cdots <t_{n-1}<T\le t_n$ with $\min_{1\le i\le n} (t_i-t_{i-1})\ge \theta$ and $n\ge 1$.

\begin{prop}\label{P:4.7}
Fix any $x_0\in \cap_{j\ge 1}E^j$.   For any $T>0$, $\delta_1, \delta_2>0$, there exist  $\delta_3>0$ and  an integer $N_4\ge N_1$ such that for all $j\ge N_4$,
\begin{equation}
\IP^{x_0}\left[w_\rho\left(X^j,\delta_3, T\right)>\delta_1\right]<\delta_2, 
\end{equation}
where
\begin{equation*}
w_\rho(x,\, \delta_3,\, T):=\inf_{\{t_i\}} \max_{i} \sup_{s, t\in [t_i, t_{i-1}]} \rho(x(s), x(t)),
\end{equation*}
where $\{t_i\}$ ranges over all possible partitions of the form $0=t_0<t_1<\cdots <t_{n-1}<T\le t_n$ with $\min_{1\le i\le n} (t_i-t_{i-1})\ge \delta_3$ and $n\ge 1$.

\end{prop}

\begin{proof}
Using the same argument as that for \cite[(4.52)-(4.54)]{Lou2}, one can get for$j\in \mathbb{N}$ and any $\delta_1, \delta_3>0$,
\begin{equation}\label{14}
\IP^{x_0}\left[w_\rho\left(X^j,\delta_3, T\right)>\delta_1\right] \le  2\left(\left[\frac{T}{\delta_3}\right]+1\right) \sup_{\substack{y\in E\\ 0\le s\le \delta_3}}\IP^y\left[ \rho\left( X^j_0, X^j_s\right)\ge \frac{\delta_1}{4} \right].
\end{equation}
For the right hand side of \eqref{14}, we further have that for $\delta_3>0$ and $j\in \mathbb{N}$ satisfying  $8^{-j}<\delta_3$, 
\begin{equation}\label{15}
\sup_{0<s\le \delta_3}\IP^y\left[ \rho\left( X^j_0, X^j_s\right)\ge \frac{\delta_1}{4} \right] \le \sup_{0<s\le 8^{-j}}\IP^y\left[ \rho\left( X^j_0, X^j_s\right)\ge \frac{\delta_1}{4} \right] +\sup_{8^{-j}\le s\le \delta_3}\IP^y\left[ \rho\left( X^j_0, X^j_s\right)\ge \frac{\delta_1}{4} \right].
\end{equation} 
We first handle the second term on the right hand side above. For any  $\delta_3>0,\, t>0,\, j\in \mathbb{N}$ such that $8^{-j}<t<\delta_3$, for any $y\in E^j$, there exists some $c_1>0$ such that
\begin{eqnarray}
 \IP^y\left[\rho (y, X_t^j)\ge \frac{\delta_1}{4}   \right]  &\le &  \IP^y\left[d_j (y, X_t^j)\ge \frac{\delta_1}{4}   \right] \nonumber
\\
& \le & \sum_{\substack{x\in E^j: d_j(x, y)\ge \frac{\delta_1}{4}\\ x\neq a^*_j}}\frac{c_1}{t^{d/2}}\left(e^{-\frac{d_j(x, y)^2}{64t}}
+e^{-\frac{2^j d_j(x,y)}{4}}\right)m_j(x) \nonumber
\\
& + & \frac{c_1}{t^{d/2}}\left(e^{-\frac{\delta_1^2}{256t}}+e^{-\frac{2^j \delta_1}{16}}\right)m_j(a^*_j). \label{16}
\end{eqnarray}
For the first term on the right hand side of \eqref{16}, we first note that for a given $y\in E^j$ 
we have 
\begin{equation*}
\left\{  x\in E^j:\, d_j(x, y)\ge \frac{\delta_1}{4}, \,  x\neq a^*_j\right\}\subset \bigcup_{k=1}^\infty \left\{  x\in E^j:\, d_j(x, y)\le   \frac{\delta_1}{4}+k, \,  x\neq a^*_j    \right\}.
\end{equation*} 
In view of the definition of $k_0$ in \eqref{def-k0}, the diameter of $K$ under Euclidean distance is at most $2k_0$. Thus for all $j\ge 1$ and all $x,y\in E^j$, it holds
\begin{equation*}
d_j(x, y)\ge \rho(x,y), \quad \rho(x,y)+2k_0\ge |x-y|.
\end{equation*}
Therefore,
\begin{equation*}
\left\{  x\in E^j:\, d_j(x, y)\le  \frac{\delta_1}{4}+k, \,  x\neq a^*_j  \right\}
\subset \left\{ x\in 2^{-j}\IZ^d:\, |x-y|\le \frac{\delta_1}{4}  +k+2k_0 \right\}.
\end{equation*}
It then follows that for any $k\in \mathbb{N}$,
\begin{equation}\label{17}
\#\left\{  x\in E^j:\, d_j(x, y)\le  \frac{\delta_1}{4}+k, \,  x\neq a^*_j    \right\}\le \left(\frac{\delta_1}{4}+k+2k_0\right)^{2d}2^{jd}.
\end{equation}
Now for the first term on the right hand side of \eqref{16}, we have 
\begin{eqnarray}
&& \sum_{\substack{x\in E^j: d_j(x, y)\ge \frac{\delta_1}{4}\\ x\neq a^*_j}}\frac{c_1}{t^{d/2}}\left(e^{-\frac{d_j(x, y)^2}{64t}}
+e^{-\frac{2^j d_j(x,y)}{4}}\right)m_j(x) \nonumber
\\
&\le & \sum_{k=0}^\infty\; \sum_{\substack{x\in E^j: d_j(x, y)\le \frac{\delta_1}{4}+k\\ x\neq a^*_j}}\frac{c_1}{t^{d/2}}\left(e^{-\frac{d_j(x, y)^2}{64t}}
+e^{-\frac{2^j d_j(x,y)}{4}}\right)m_j(x) \nonumber
\\
\eqref{17} &\le & \sum_{k=0}^\infty \frac{c_1}{t^{d/2}}\left( e^{-\frac{(\delta+k)^2}{256t}} +e^{-\frac{2^j(\delta+k)}{16}} \right)\cdot (\delta_1+k+2k_0)^{2d}2^{jd} 2^{-jd}\nonumber
\\ 
&\le & \sum_{k=0}^\infty \frac{c_1(\delta_1+k+2k_0)^{2d}}{t^{d/2}}e^{-\frac{(\delta_1+k)^2}{256t }}+\sum_{k=0}^\infty \frac{c_1(\delta_1+k+2k_0)^{2d}}{t^{d/2}} e^{-\frac{2^j (\delta_1+k)}{16}}. \label{18-2}
\end{eqnarray}
Now, for the first term  on the right hand side of \eqref{18-2}, for any $8^{-j}<t<\delta_3\le T$,
\begin{eqnarray}
\sum_{k=0}^\infty \frac{c_1(\delta_1+k+2k_0)^{2d}}{t^{d/2}}e^{-\frac{(\delta_1+k)^2}{256t }} & \le & e^{-\frac{\delta_1^2}{512\delta_3}}\left(  \sup_{t\in (0, T]}\frac{c_1}{t^{d/2}}e^{-\frac{\delta_1^2}{1024T}}\right)\sum_{k=0}^\infty (\delta_1+k+2k_0)^{2d}e^{-\frac{(\delta_1+k)^2}{1024T}} \nonumber
\\
&\le & c_2(\delta_1, T, k_0)e^{-\frac{\delta_1^2}{512\delta_3}}. \label{19}
\end{eqnarray}
For the second term on the right hand side of \eqref{18-2}, noticing that $t\ge 8^{-j}$,
\begin{eqnarray}
\sum_{k=0}^\infty \frac{c_1(\delta_1+k+2k_0)^{2d}}{t^{d/2}} e^{-\frac{2^j (\delta_1+k)}{16}}  &\le &  c_1\left( (2^j)^{3d/2}e^{-\frac{2^j\delta_1}{32}} \right)\sum_{k=0}^\infty (\delta_1+k+2k_0)^{2d}e^{-\frac{\delta_1+k}{32}} \nonumber
\\
&\le & \left( (2^j)^{\frac{3d}{2}}e^{-\frac{2^j\delta_1}{32}} \right)\cdot c_3(\delta_1, k_0). \label{21}
\end{eqnarray}
Now we deal with the second term on the right hand side of \eqref{16}.   For any $t<T$ where $\delta_3>0$ is a generic constant,
\begin{eqnarray}
&& \frac{c_1}{t^{d/2}}\left(e^{-\frac{\delta_1^2}{256t}}+e^{-\frac{2^j \delta_1}{16}}\right)m_j(a^*_j) \nonumber
\\
(\text{Proposition }\ref{P1}) & \le &  \frac{c_4}{t^{d/2}}\left(e^{-\frac{\delta_1^2}{256t}}+e^{-\frac{2^j \delta_1}{16}}\right) 2^{-j} \nonumber
\\
& \le & c_4\cdot e^{-\frac{\delta_1^2}{512\delta_3}}\left(\sup_{t\in (0, T]}\frac{1}{t^{d/2}}e^{-\frac{\delta_1^2}{512t}}\right)+c_4\cdot  (2^j)^{\frac{3d}{2}-1} e^{-\frac{2^j \delta_1}{16}} \nonumber
\\
&\le & c_5(\delta_1, T)e^{-\frac{\delta_1^2}{512\delta_3}}+c_4\cdot  (2^j)^{\frac{3d}{2}-1} e^{-\frac{2^j \delta_1}{16}}. \label{20}
\end{eqnarray}
From here, first we replace the two terms on the right hand side of \eqref{18-2} with \eqref{19} and \eqref{21}, then plug what is obtained as well as \eqref{20} into the right hand side of \eqref{16}. Consolidating the common terms gives us
that for any pair of $\delta_3>0,\, j\in \mathbb{N}$ such that $8^{-j}<\delta_3$,
\begin{eqnarray}
\sup_{\substack{y\in E^j\\ t\in [8^{-j}, \delta_3]}} \IP^y\left[\rho (y, X_t^j)\ge \frac{\delta_1}{4}   \right]  &\le &   c_6(\delta_1, T, k_0)e^{-\frac{\delta_1^2}{512\delta_3}}+c_7(\delta_1, k_0)\left( (2^j)^{\frac{3d}{2}}e^{-\frac{2^j\delta_1}{32}} \right).
\end{eqnarray}
From the above, given any pair of $\delta_1, \delta_2>0$,  we can first choose $\delta_3>0$ sufficiently small such that 
\begin{equation}\label{chosen-delta3}
c_6(\delta_1, T, k_0)e^{-\frac{\delta_1^2}{512\delta_3}}<\frac{\delta_2\delta_3}{4(T+2\delta_3)}.
\end{equation}
With this $\delta_3$  chosen  above, then we choose $n_1>0$ satisfying $8^{-n_1}<\delta_3$ such that for all $j\ge n_1$,
\begin{equation*}
c_7(\delta_1, k_0)\left( (2^j)^{\frac{3d}{2}}e^{-\frac{2^j\delta_1}{32}} \right)<\frac{\delta_2\delta_3}{4(T+2\delta_3)}.
\end{equation*}
Thus given any pair of $\delta_1, \delta_2>0$, there exists $\delta_3>0$ and $n_1\in \mathbb{N}$ with $8^{-n_1}<\delta_3$ such that for all $j\ge n_1$, 
\begin{equation}\label{22}
\sup_{\substack{y\in E^j,\\8^{-j}\le t\le \delta_3}}\IP^y\left[ \rho\left( X^j_0, X^j_t\right)\ge \frac{\delta_1}{4} \right]<\frac{\delta_2\delta_3}{2(T+2\delta_3)}.
\end{equation}
The same argument as that for \cite[(4.50)]{Lou2} yields that given any $\delta_1, \delta_2>0$ and the $\delta_3>0$ therefore selected in \eqref{chosen-delta3}, there exists $n_2$ such that for all $j\ge n_2$,
\begin{equation}\label{23}
\sup_{\substack{y\in E^j,\\0<s\le 8^{-j}}}\IP^y\left[ \rho\left( X^j_0, X^j_s\right)\ge \frac{\delta_1}{4} \right]<\frac{\delta_2\delta_3}{2(T+2\delta_3)}.
\end{equation}
Finally, plugging \eqref{22} and \eqref{23} into the right hand side of \eqref{14} yields that given $\delta_1, \delta_2>0$, there exist $\delta_3>0$ and $n_1, n_2\in \mathbb{N}$ such that for all $j\ge n_1\vee n_2$,
\begin{equation*}
\IP^{x_0}\left[w_\rho\left(X^j,\delta_3, T\right)>\delta_1\right] \le  2\left(\left[\frac{T}{\delta_3}\right]+1\right) \cdot \frac{\delta_2\delta_3}{T+2\delta_3} \le 2\cdot \frac{T+2\delta_3}{\delta_3} \cdot \frac{\delta_2\delta_3}{T+2\delta_3}<2\delta_2.
\end{equation*}
\end{proof}

\begin{thm}\label{T:tightness}
Fix $x_0\in \bigcap_{j\ge 1}E^j$. For every     $T>0$, the laws of $\{X^j, \IP^{x_0}, j\ge 1\}$ are C-tight in the Skorokhod space $\mathbf{D}([0, T],   E, \rho)$ equipped with the Skorokhod topology. 
\end{thm}
\begin{proof}
This follows immediately from \cite[Chapter VI, Proposition 3.21]{JS}. In view of Propositions \ref{P:4.6}-\ref{P:4.7}.
\end{proof}

\begin{rmk}\label{R:4.10}
By the same proof as that to Theorem \ref{T:tightness}, it can be shown that given any $T>0$, the laws of $\{X^j, \IP^{a^*_j}, j\ge 1\}$ are C-tight in the Skorokhod space $\mathbf{D}([0, T],   E, \rho)$ equipped with the Skorokhod topology. 
\end{rmk}

Before proving the next lemma, we define the following class of functions:
\begin{align}
\mathcal{G}:&=\{ f\in C_c^3(\IR^d),\, f=\text{constant on }K. \}.\label{def-class-G}
\end{align}
For $f\in \mathcal{G}$,  we define
\begin{equation}\label{def-wt-Lk}
\mathcal{L}^j f(x):=2^{2j}\sum_{y\leftrightarrow x\text{ in }E^j}\left(f(y)-f(x)\right)\frac{1}{v_j(x)}, \quad \text{ for }x\in E^k_0.
\end{equation}
We also  set 
\begin{equation}\label{def-S0k}
S^j:=\{x\in E^j:\; x\neq a^*_j,\, v_j(x)=2d \text{ in } G^j \},
\end{equation}
where $G^j$ has been defined in \S\ref{Intro}.
\begin{lem}\label{L:4.9}
 For every $\delta>0$ and every $f\in \mathcal{G}$, there exists some $n_{\delta, f}\in \mathbb{N}$, such that for all $j\ge n_{\delta, f}$:
 \begin{description}
\item{(i)} \, $m_j(E^j\backslash S^j)<\delta$;
 \item{(ii)}  As $j\rightarrow \infty$,  $\mathcal{L}^jf$ converges uniformly to 
\begin{equation}\label{def-L}
\mathcal{L}f(x):=\frac{1}{2d}\Delta f(x) +O(1)2^{-j}  \quad \text{on } S^{n_{\delta, f}}.
\end{equation}
 \end{description}
Also, there exists some constant $C_7>0$ independent of $j$ such  that for all $j\ge 1$,
\begin{equation*}
\mathcal{L}^j f(x)\le C_7, \quad \text{ for all }x\in E^j. 
\end{equation*}
\end{lem}

\begin{proof}
To claim (i), we notice that $\{E^j\backslash S^j\}\subset \{x=a^*_j \text{ or }x\leftrightarrow a^*_j\}$. Thus by Proposition \ref{P1}, there exists $c_1>0$ such that for all $j\ge j_0$ specified in Proposition \ref{P1},
\begin{eqnarray*}
m_j(E^j\backslash S^j)\le  m_j(\{x=a^*_j \text{ or }x\leftrightarrow a^*_j\})   \le m_j(a^*_j)+v_j(a^*_j)\cdot 2^{-jd}\le c_1 \cdot 2^{-j}.
\end{eqnarray*}
Using Taylor's expansion we have  for any $f\in \mathcal{G}$ and any $j\ge j_0$,
\begin{eqnarray}
\mathcal{\wt{L}}_{j} f(x) & = & 2^{2j}\sum_{y\leftrightarrow x }\Bigg[\sum_{i=1}^d \frac{\partial f(x)}{\partial x_i} (y_i-x_i)+\frac{1}{2}\sum_{i, l=1}^d\frac{\partial^2 f(x)}{\partial x_i\partial x_l} (y_i-x_i)(y_l-x_l) +O(1)|y-x|^3\Bigg]\frac{1}{v_j(x)}. \nonumber
\end{eqnarray}
Hence
\begin{equation}\label{26}
\mathcal{L}^jf(x)=\frac{1}{2d}\Delta f(x)+O(1)2^{-j}, \quad \text{for }x\in S^j.
\end{equation}
Since $\{S^j\}_{j\ge 1}$ is an increasing sequence of sets in $j$, both (i) and (ii) have been justified.   To justify the last claim, we first note that by \eqref{26} and the fact that $f\in C_c^3(\IR^d)$, there exists $c_2>0$ independent of $j$ such that 
\begin{equation}\label{27}
\mathcal{L}^j f(x)\le c_2, \quad \text{for all }x\in \bigcup_{j\ge 1}S^j.
\end{equation}
We denote by $c_K:= f|_K$. Given any $x \leftrightarrow a^*_j$ in $E^j$, by the definition of $G^j$ in \eqref{BMD-2}, there must  exist a point $a\in K$ such that $|x-a|\le 2^{-j}$. Since $f\in C^3_c(\IR^d)$ and is constant on $K$, the first order derivatives of $f$ vanish on $K$.  By Taylor expansion, there exists some constant $c_3>0$ only depending on $f$ but not $j$ such that 
\begin{equation}\label{25}
\left( f(x)-c_K \right)=\left| f(x)-f(a)  \right| \le \sum_{i=1}^d \left| \frac{\partial f}{\partial x_i} (a)  \right|\cdot \left|x-a\right|+c_2\cdot |x-a|^2 \le  c_3\cdot 2^{-2j}.
\end{equation}
Thus  by the definition of $\mathcal{L}^j$ and Proposition \ref{P1},
\begin{equation} \label{40}
\mathcal{L}^j f(a^*_j)=2^{2j}\sum_{x\leftrightarrow a^*_j} \left( f(x)-c_K \right)\cdot v_j(a^*_j)^{-1}\le 2^{2j}\cdot c_3\cdot 2^{-2j}\le c_3. 
\end{equation}
To bound $\mathcal{L}^j f(x)$ for $x\leftrightarrow a^*_j$, we first note in this case,
\begin{eqnarray}\label{24}
\mathcal{L}^j f(x)=2^{2j}\sum_{y\leftrightarrow x}\left( f(y)-f(x)\right)\cdot v_j(x)^{-1} \le 2^{2j}\max_{x: x\leftrightarrow a^*_j} \max_{y: y \leftrightarrow x} \left|f(y)-f(x)  \right|.
\end{eqnarray}
For each $y$ in the second $``\max"$ in \eqref{24}, notice that there exists $a\in K$ such that $|y-a|\le 2\cdot 2^{-j}$. Thus by similar reasoning for \eqref{25}, $|f(y)-c_K|\le c_3\cdot (2\cdot 2^{-j})^2$. Since for all $x$ in the first $``\max"$ in \eqref{24}, $|f(x)-c_K|\le c_3\cdot 2^{-2j}$, by triangle inequality it follows that for some $c_4>0$ independent of $j$, for all $j\ge 1$, $\mathcal{L}^j f(x) \le c_4$ for all $x$ such that $x\leftrightarrow a^*_j$. This together with \eqref{40} shows 
\begin{equation*}
\mathcal{L}^j f(x) \le c_3 \vee c_4,\quad \text{for all }x\in \bigcup_{j\ge 1}(E^j\backslash S^j).
\end{equation*}
This combined with \eqref{27} proves the last claim of this lemma. The proof is complete.
\end{proof}

The following lemma is used in the proof of the main result: Theorem \ref{T:main-result}.
\begin{lem}\label{L:4.11}
Fix $T>0$.  Given any $\delta>0$, there exist  $C_8>0$ and an integer $N_\delta \ge N_1$,  such that for all $j\ge N_\delta$, 
\begin{equation*}
\sup_{t\in [(2^j\delta)^{-2/d}, T]}\IP^{x_0}\left[ X^j_t \notin S^j  \right]\le C_8\delta,
\end{equation*}
\end{lem}

\begin{proof}
Given any $\delta>0$, choose $j_\delta\in \mathbb{N}$ large enough such that $(2^{j_\delta}\delta)^{-2/d}<T$. For any  $t\in [(2^j\delta)^{-2/d}, \,T]$ with $j\ge j_\delta$, by Corollary \ref{HKUB} and Proposition \ref{P1}, there exists $c_1, c_2>0$ such that
\begin{eqnarray}
\IP^{x_0}\left[ X_t^j \notin S^j \right] \nonumber
 & \le & \sum_{y\notin S^j} \frac{c_1}{t^{d/2}}  m_j(y) \nonumber
\\
&\le & \sum_{x\leftrightarrow a^*_j} \left( \frac{c_1}{t^{d/2}}\cdot m_j(x)\right)+\frac{c_1}{t^{d/2}}\cdot m_j(a^*_j)  \nonumber
\\
(\text{Proposition }\ref{P1}) &\le & c_2\cdot 2^{j(d-1)}\cdot c_1 t^{-d/2}\cdot 2^{-jd}+c_1 t^{-d/2}\cdot c_2\cdot 2^{j(d-1)}\cdot 2^{-jd} \nonumber
\\
(t>(2^j\delta)^{-2/d}) &\le & 2c_1c_2\cdot \delta. \nonumber
\end{eqnarray}
This proves the desired result by letting $N_\delta$ be the selected $j_\delta$.
\end{proof}

\begin{thm}\label{T:main-result}
Given $x_0\in \cap_{j\ge 1}E^j$,
$\{X^j,\IP^{x_0},\, j\ge 1\}$ converges weakly to the BMD described in \eqref{def-BMD}  starting from $x_0$.
\end{thm}

\begin{proof}
Since the laws of $\{X^j\}_{j\ge 1}$ are C-tight in $\mathbf{D}([0, T], E, \rho)$, any sequence has a weakly convergent subsequence supported on the set of continuous paths. Denote by $\{X^{j_l}: l\ge 1\}$ any such  weakly convergent subsequence, and denote by $Y$ its weak limit which is  continuous. By  Skorokhod representation theorem (see, e.g.,  \cite[Chapter 3, Theorem 1.8]{EK}), we may assume that $\{X^{j_l},l\ge 1\}$ as well as $Y$ are defined on a common probability space $(\Omega, \mathcal{F}, \IP)$, so that  $\{X^{j_l},l\ge 1\}$  converges almost surely to $Y$ in the Skorokhod topology.

For every $t\in [0. T]$, we set $\mathcal{M}_t^{j_l} :=\sigma(X^{j_l}_s, s\le t)$  and $\mathcal{M}_t :=\sigma(Y_s, s\le t)$. It is obvious that $\mathcal{M}_t\subset \sigma\{\mathcal{M}_t^{k_j}: j\ge 1\}$. By the same argument as that for \cite[Theorem 5.3]{Lou2} with the use of Lemma 5.2 in \cite{Lou2} being replaced with Lemma \ref{L:4.11}, it can be shown that  $(Y, \IP^{x_0})$ is indeed a solution to the   $\mathbf{D}([0, T], E, \rho)$  martingale problem $(\mathcal{L}, \mathcal{G})$   with respect to the filtration $\{\mathcal{M}_t\}_{t\ge 0}$.

Next we claim that the BMD associated with the Dirichlet form described by \eqref{def-BMD}  is a strong Feller process. To see this, we denote by $\{G^\alpha\}_{\alpha>0}$ the resolvent operators of $X$, and denote by $\{G^\alpha_{E\backslash \{a^*\}}\}_{\alpha>0}$ the resolvent operators of $X^{E\backslash \{a^*\}}$, the part process of $X$ killed hitting $a^*$, which has the same distribution as regular Brownian motion on $\IR^d$ killed upon hitting $K$.  By strong Markov property, for $x\in E$, for every bounded measurable function $f: E\rightarrow \IR$,
\begin{eqnarray}
G^\alpha f(x) &= & G^\alpha_{E\backslash \{a^*\}}+G^\alpha f(a^*)\cdot\int_0^\infty e^{-\alpha s}\; \IP^x [\sigma_{\{a^*\}}\in ds]
\end{eqnarray}
In the right hand side above, the map $x\mapsto \IP^x [\sigma_{\{a^*\}}\in ds]$ is continuous because every point on the boundary of $K\subset \IR^d$ is regular for $K$. Therefore we concludee that $X$ is a strong Feller process because $G^\alpha(b\mathcal{B}(E))\subset C(E)$. In view of \cite[\S1.5]{Chen1}, the infinitesimal generator of $X$ can be described by $(\mathcal{L}, \mathcal{D(L)} )$, where $u\in \mathcal{F}$ is in $\mathcal{D(L)}$ if there exists $f\in L^2(E)$ such that
\begin{eqnarray}
\frac{1}{2}\int_{E\backslash \{a^*\}} \nabla u(x)\nabla v(x) dm=  \int_{E\backslash \{a^*\}} f(x)\cdot v(x)dm, \quad \text{for all }v\in \mathcal{F}.
\end{eqnarray}
It also holds that $\mathcal{L}u=f=\frac{1}{2}\Delta u$ for $u\in \mathcal{D(L)}$. It then is clear that the bp-closure of the graph of $\mathcal{(L, D(L))}$ is contained in the bp-closure of the graph of $\mathcal{(L, G)}$. By \cite[Lemema 3.4.18]{AB}, any solution to the martingale problem $(\mathcal{L}, \mathcal{G})$ must be a solution to the martingale problem $\mathcal{(L, G)}$.  Since $X$ is a strong Feller process, the martingale problem $\mathcal{(L, D(L))}$ must be well-posed with its unique solution being $X$. Therefore the martingale problem $\mathcal{(L, G)}$ must be well-posed with its unique solution being $X$. This means that  $X$ is the sequencial limit of any weakly convergent subsequence of $\{X^j\}_{j\ge 1}$, the proof is complete. 
\end{proof}

\begin{rmk}
In view of Remark \ref{R:4.10}, the same argument can show that  $\{X^j,\IP^{a^*_j},\, j\ge 1\}$ converges weakly to the BMD given by \eqref{def-BMD}  starting from $a^*$.
\end{rmk}

\vskip 0.3truein

\noindent {\bf Shuwen Lou}

\smallskip \noindent
Department of Mathematics and Statistics, Loyola University Chicago,
\noindent
Chicago, IL 60660, USA

\noindent
E-mail:  \texttt{slou1@luc.edu}


\begin{thebibliography}{99}

\bibitem{MB} M. T. Barlow, {\it Random walks and heat kernels on graphs.} Cambridge University Press, 2017.

\bibitem{AB} A. Bovier, {\it Markov Processes}. Available at \url{https://wt.iam.uni-bonn.de/fileadmin/WT/Inhalt/people/Anton_Bovier/markov-processes/200910_WS/wt3-bonn.pdf}.



\bibitem{CKS} E. A. Carlen,  S. Kusuoka and D. W. Stroock,
Upper bounds for symmetric Markov transition functions, {\it Ann. Inst. H. Poincar\'{e}. Probab. Statist.
\bf 23} (1987) 245-287.

\bibitem{Chen1} Z.-Q. Chen, {\it Brownian motion with darning}. Available at
\url{https://sites.math.washington.edu/~zchen/BMD_lecture.pdf}


\bibitem{CF} Z.-Q. Chen and M. Fukushima,   {\it Symmetric Markov Processes, Time Change and Boundary
Theory}.   Princeton University Press, 2011.

\bibitem{CFKZ} Z.-Q. Chen, P. J. Fitzsimmons, K. Kuwae and T.-S. Zhang, Stochastic calculus for symmetric Markov processes. {\it Ann. Probab. \bf 36} (2008), 931-970. 

\bibitem{CL} Z. -Q. Chen and S. Lou, Brownian motion on some spaces with varying dimension. {\it Ann. Probab. \bf 47} (2019), 213-269. 


\bibitem{ChungZhao} K. L. Chung and Z. Zhao, {\it From Brownian motion to Schr\"{o}dinger's equation}. Springer, 2001.

\bibitem{EK} S. N. Ethier and T. G. Kurtz, {\it Markov Processes: Characterization and Convergence}. Wiley, New York, 1986. 

\bibitem{FOT} M. Fukushima, Y. Oshima, and M. Takeda, {\it Dirichlet Forms and Symmetric Markov Processes}, Second revised and extended edition.  de Gruyter Studies in Mathematics, vol. 19, Walter de Gruyter \& Co., Berlin, 2011.



\bibitem{JS} J. Jacob and A. N. Shiryaev, {\it Limit Theorems for Stochastic Processes}. Springer, Berlin, 1987. 

\bibitem{Lou1} S. Lou, Discrete approximate to Brownian motion with varying dimension in bounded domains.    To appear in {\it Kyoto J. Math.}  Available at \url{https://arxiv.org/pdf/2007.01933.pdf}.



\bibitem{Lou2} S. Lou, Discrete approximate to Brownian motion with varying dimension in unbounded domains.   Available at \url{https://arxiv.org/pdf/2110.12716.pdf}.


\bibitem{MF} F. Markus,  {\it Markov processes and martingale problems}. Available at \url{https://www.math.unipd.it/~fischer/Didattica/MarkovMP.pdf}.


\end{thebibliography}
 \end{document}